\documentclass[10pt]{amsart}

\usepackage{amsmath,amsthm,amssymb}
\usepackage[all,cmtip]{xy}
\usepackage{todonotes}
\usepackage{url}
\usepackage{framed}

\newtheorem{thm}{Theorem}
\newtheorem{cor}[thm]{Corollary}
\newtheorem{lem}[thm]{Lemma}
\newtheorem{prop}[thm]{Proposition}
\newtheorem{defn}[thm]{Definition}

\theoremstyle{remark}
\newtheorem{remark}[thm]{Remark}
\newtheorem{example}[thm]{Example}

\numberwithin{equation}{section}
\numberwithin{thm}{section}

\DeclareMathOperator{\Hom}{Hom}

\DeclareMathOperator{\Aut}{Aut}

\DeclareMathOperator{\Spec}{Spec}
\DeclareMathOperator{\Proj}{Proj}

\DeclareMathOperator{\NS}{NS}

\DeclareMathOperator{\chr}{char}
\DeclareMathOperator{\SL}{SL}

\DeclareMathOperator{\rank}{rank}

\newcommand{\R}{\mathbb{R}}

\newcommand{\C}{\mathbb{C}}
\newcommand{\Q}{\mathbb{Q}}
\newcommand{\Z}{\mathbb{Z}}

\newcommand{\PP}{\mathbb{P}}

\newcommand{\abs}[1]{\left\vert#1\right\vert}
\newcommand{\set}[1]{\left\{#1\right\}}

\newcommand{\inn}[1]{\left\langle#1\right\rangle}

\title[The Picard number of K3 surfaces of BHK type]{A simple formula for the Picard number of K3 surfaces of BHK type}

\author{Christopher Lyons}
\address{California State University, Fullerton \\
Department of Mathematics \\
800 N. State College Blvd \\
Fullerton, CA 92834 }
\email{clyons@fullerton.edu}

\author{Bora Olcken}
\address{University of Calfornia, Santa Cruz \\
Mathematics Department \\
1156 High Street \\
Santa Cruz, CA 95064}
\email{bolcken@ucsc.edu}


\begin{document}

\maketitle

\begin{abstract}
The BHK mirror symmetry construction stems from work Berglund and H\"{u}bsch \cite{BH}, and applies to certain types of Calabi-Yau varieties that are birational to finite quotients of Fermat varieties.  Their definition involves a matrix $A$ and a certain finite abelian group $G$, and we denote the corresponding Calabi-Yau variety by $Z_{A,G}$.  The transpose matrix $A^T$ and the so-called dual group $G^T$ give rise to the BHK mirror variety $Z_{A^T,G^T}$.  In the case of dimension 2, the surface $Z_{A,G}$ is a \emph{K3 surface of BHK type}.

Let $Z_{A,G}$ be a K3 surface of BHK type, with BHK mirror $Z_{A^T,G^T}$.  Using work of Shioda \cite{Shioda}, Kelly shows in \cite{Kelly} that the geometric Picard number $\rho(Z_{A,G})$ of $Z_{A,G}$ may be expressed in terms of a certain subset of the dual group $G^T$.  We simplify this formula significantly to show that $\rho(Z_{A,G})$ depends only upon the degree of the mirror polynomial $F_{A^T}$.
\end{abstract}

\section{Introduction}

Let $k$ be an algebraically closed field of characteristic $p\geq 0$, and let $X$ be a smooth projective surface defined over $k$.  The Picard number of $X$ is defined as the rank of the Neron-Severi group of $X$:
\[
\rho = \rho(X) = \rank_\Z \NS(X),
\]
where $\NS(X)$ denotes the group of divisors on $X$ modulo algebraic equivalence.  The Picard number is a fundamental invariant of a surface, but in practice the determination of $\rho(X)$ can be difficult for a given surface $X$.

The case of K3 surfaces illustrates this point well, where one has $1\leq \rho\leq 20$ in characteristic 0 and $1\leq \rho\leq 22$ in positive characteristic.  For instance, when $\rho(X)\geq 5$, one knows that $X$ has an elliptic fibration, and when $\rho(X)\geq 12$, then such a fibration exists having a section (see for instance \cite[\S11.1]{Huy}).  If $\rho(X)\geq 17$, one may often deduce the existence of a nontrivial algebraic correspondence between $X$ and an abelian surface \cite{Mor}.  K3 surfaces satisfying $\rho = 22$ (in positive characteristic) fall into the highly interesting class of \emph{supersingular} K3 surfaces.   This small selection of facts shows the importance of determining the Picard number of a K3 surface, but this task is usually far from straightforward.  This may be inferred, for instance, from van Luijk's construction \cite{vL} of the first explicit example of a K3 surface over $\Q$ with $\rho=1$, a relatively recent feat within the much longer history of K3 surfaces.

In \cite{Shioda}, Shioda identifies a special class of surfaces for which one the Picard number may be determined in a combinatorial manner.  Any one of these surfaces $X$ is birational to the quotient of a Fermat surface $X_0^d+X_1^d+X_2^d+X_3^d=0$ by a finite group $\Gamma$, and this allows Shioda to obtain a formula of the form 
\[
\rho(X) = b_2(X)-\#(\mathfrak I_d(p)\cap H),
\]
where $b_2(X)$ is the second Betti number of $X$, $H$ is a certain subgroup of $(\Z/d)^4$ (determined by $\Gamma$), and $\mathfrak I_d(p)$ is a certain subset of $(\Z/d)^4$ (see \S\ref{Pic-no-sec}).  In this sense, the formula for $\rho(X)$ in \cite{Shioda} is elementary and computable; however, the nontrivial amount of calculation required generally makes it difficult to predict the value of $\rho(X)$ without completing the full computation.

One interesting class of surfaces to which Shioda's formulas apply are \emph{K3 surfaces of BHK type}.  These are the 2-dimensional examples of special types of Calabi-Yau varieties for which a notion of mirror symmetry was initiated by Berglund--H\"{u}bsch \cite{BH}, and extended by Berglund--Henningson \cite{BHen} and Krawitz \cite{Krawitz}.  Roughly speaking, a K3 surface of BHK type is determined by what we shall call a \emph{adequate BHK pair} $(A,G)$, where $A=(A_{ij})$ is a 4-by-4 matrix of nonnegative integers and $G$ is a subgroup of $(\Z/d)^4$ (for some $d\geq 1$ determined by $A$); see \S\ref{const-sec} and \S\ref{K3-sec} for the precise requirements on $A$ and $G$.  The matrix $A$ yields a polynomial of the form
\begin{equation}\label{FA-eqn}
F_A = \sum_{i=0}^3 x_0^{A_{i0}}x_1^{A_{i1}}x_2^{A_{i2}}x_3^{A_{i3}},
\end{equation}
one which turns out to be quasihomogeneous of some degree $h$ and weights ${\bf q} = (q_0,q_1,q_2,q_3)$.  This will define a hypersurface $X_A$ of degree $h$ in the weighted projective space $\PP^3({\bf q})$, and its minimal resolution $\tilde X_A$ will be a K3 surface.  The group $G$ yields a group of symplectic automorphisms on $\tilde X_A$, and hence the minimal resolution of the quotient is again a K3 surface $Z_{A,G}$; the latter surface is called a K3 surface of BHK type.

The BHK mirror construction starts by forming the transpose $A^T$, which in turn yields a second polynomial $F_{A^T}$.  From $G$, one may also form the so-called \emph{dual group} $G^T\subseteq (\Z/d)^4$ (see \S\ref{mirror-sec}).  When $(A^T,G^T)$ is another adequate BHK pair, it gives rise to the K3 surface $Z_{A^T,G^T}$, which is the \emph{BHK mirror surface} of $Z_{A,G}$.

In \cite{Kelly}, Kelly applies Shioda's formula to the BHK construction to reveal a striking relationship between the Picard numbers of the K3 surfaces $Z_{A,G}$ and $Z_{A^T,G^T}$, one that adds weight to the ``mirror'' adjective:

\begin{thm}[Kelly]\label{Kelly-intro-thm}
Let $Z_{A,G}$ and $Z_{A^T,G^T}$ be K3 surfaces of BHK type that are BHK mirrors.  Their Picard numbers are given by
\begin{eqnarray*}
\rho(Z_{A,G}) &=& 22-\#(\mathfrak I_d(p)\cap G^T) \\
\rho(Z_{A^T,G^T}) &=& 22-\#(\mathfrak I_d(p)\cap G).
\end{eqnarray*}
\end{thm}

Yet despite the appealing the relation in Theorem \ref{Kelly-intro-thm}, the precise values of $\rho(Z_{A,G})$ and $\rho(Z_{A^T,G^T})$ can still be difficult to anticipate without a fair amount of calculation.  One arithmetically interesting demonstration of this is given by the following question: Noting that $F_A$ has integral coefficients, if we fix the pair $(A,G)$ and allow the prime $p$ to vary, how does $\rho(Z_{A,G})$ vary?   As an intriguing example at the end of \cite{Kelly} suggests, the answer is not obvious, but nevertheless appears to be simpler than one might expect from Shioda's formula.
 
Our main result simplifies the formulas in Theorem \ref{Kelly-intro-thm} and also answers the question of how the Picard number of $Z_{A,G}$ varies with $p$:

\begin{thm}\label{main-intro-thm}
Let $Z_{A,G}$ and $Z_{A^T,G^T}$ be K3 surfaces of BHK type that are BHK mirrors.  Let $h$ (resp.\ $h_T$) denote the degree of the quasihomogeneous polynomial $F_A$ (resp.\ $F_{A^T}$).  The geometric Picard numbers of these surfaces are given as follows:
\begin{enumerate}
\item If $\chr k = 0$ then
\begin{eqnarray*}
\rho(Z_{A,G}) &=& 22-\phi(h_T) \\
\rho(Z_{A^T,G^T}) &=& 22-\phi(h).
\end{eqnarray*}
\item If $\chr k=p>0$ with $p\nmid d$ then
\begin{eqnarray*}
\rho(Z_{A,G}) &=& \begin{cases}
22 & \text{ if } p^\ell\equiv -1 \pmod{h_T} \text{ for some } \ell \\
22-\phi(h_T) & \text{ if } p^\ell\not\equiv -1 \pmod{h_T} \text{ for all } \ell \\
\end{cases} \\
\rho(Z_{A^T,G^T}) &=& \begin{cases}
22 & \text{ if } p^\ell\equiv -1 \pmod{h} \text{ for some } \ell \\
22-\phi(h) & \text{ if } p^\ell\not\equiv -1 \pmod{h} \text{ for all } \ell \\
\end{cases}
\end{eqnarray*}
\end{enumerate}
\end{thm}

In particular, we recover a result of Inose \cite[Prop.\ 1.1]{Inose} in the special case of K3 surfaces of BHK type:

\begin{cor}
The Picard number of $Z_{A,G}$ is independent of $G$.
\end{cor}

Another consequence is that the possible values of Picard numbers of K3 surfaces of BHK type are constrained only by the values of the $\phi$-function.  Indeed, by looking through the wealth of examples in \cite{ABS, CLPS, CP}, and also noting that $\phi(n)\neq 14$ for any $n\geq 1$, we can conclude:

\begin{cor}
In characteristic 0 (resp.\ $p>0$), the Picard number of a K3 surface of BHK type may equal any even integer between $2$ and $20$ (resp. $22$), with the exception of $8$.
\end{cor}

The proof of Theorem \ref{main-intro-thm} proceeds by an analysis of the set $\mathfrak I_d(p)\cap G \subseteq (\Z/d)^4$, for a given BHK pair $(A,G)$.  For a certain subset $\mathfrak A_d\subseteq (\Z/d)^4$ (see \S\ref{Pic-no-sec}), the multiplicative group $(\Z/d)^\times$ acts upon $\mathfrak A_d\cap G$ by coordinatewise multiplication, and one may recast $\mathfrak I_d(p)\cap G \subseteq (\Z/d)^4$ as the union of those orbits of this action which possess a certain property.   (This property of the orbits is described in terms of certain special elements in $\mathfrak A_d$; see  \S\ref{orbit-sec}.  We note that this idea is also implicit in \cite[Cor.\ 2]{Shioda}.)  On the other hand, $\mathfrak I_d(p)\cap G$ also has a Hodge theoretic meaning, and in the context of K3 surfaces this allows one to deduce that there is at most one orbit having the aforementioned property (see Proposition \ref{exactly-one-prop}).  An explicit description of this one orbit then gives the theorem.

In the final section of the paper, we revisit the example from the end of \cite{Kelly}.

\section*{Acknowledgements}
We thank Tyler Kelly and Nathan Priddis for many helpful insights and suggestions on an earlier version of this paper.  The second author was partially supported by the 2015 Math Summer Research Program at California State University, Fullerton.

\section{BHK pairs}\label{const-sec}

In this section we fix $p$ to be either 0 or a (positive) prime integer.

\begin{defn}\label{A-defn}
Let $A=(A_{ij})_{0\leq i,j\leq 3} \in M_{4\times 4}(\Z)$.  We will say $A$ is a \emph{weighted Delsarte matrix} if it satisfies all of the following:
\begin{enumerate}
\item\label{A-pos} All entries of $A$ are nonnegative
\item\label{A-zero} Each row has at least one zero\item\label{A-det} We have $p\nmid \det(A)$ (and in particular $\det(A)\neq 0$)
\item\label{A-weights} The vector $A^{-1}(1,1,1,1)^T$ consists of positive entries
\end{enumerate}
\end{defn}

Now let $k$ denote an algebraically closed field of characteristic $p$.  Associated to $A$, we define the polynomial
\[
F_A = \sum_{i=0}^3 \prod_{j=0}^3 x_j^{A_{ij}} = \sum_{i=0}^3 x_0^{A_{i0}}x_1^{A_{i1}}x_2^{A_{i2}}x_3^{A_{i3}}.
\]
Notice that if $\lambda\in k^\times$ and ${\bf n} = (n_0,n_1,n_2,n_3)\in \Z^4$, then
\begin{eqnarray}
F_A(\lambda^{n_0}x_0,\lambda^{n_1}x_1,\lambda^{n_2}x_2,\lambda^{n_3}x_3) &=& \sum_{i=0}^3 \prod_{j=0}^3 (\lambda^{n_j}x_j)^{A_{ij}} \\
 &=& \sum_{i=0}^3 \left(\lambda^{\sum_{j=0}^3 A_{ij}n_j}\prod_{j=0}^3 x_j^{A_{ij}}\right) \\
 &=& \sum_{i=0}^3 \left(\lambda^{(A {\bf n}^T)_i}\prod_{j=0}^3 x_j^{A_{ij}}\right). \label{trans}
\end{eqnarray}
As a first consequence of this, let the vector in (\ref{A-weights}) be given by
\begin{equation}\label{eq1}
A^{-1}(1,1,1,1)^T = \left(\frac{q_0}{h},\frac{q_1}{h},\frac{q_2}{h},\frac{q_3}{h}\right),
\end{equation}
where each $q_i$ and $h$ are (positive) integers and $\gcd(q_0,q_1,q_2,q_3)=1$.  Putting ${\bf q}=(q_0,q_1,q_2,q_3)$, we have $A{\bf q}^T = (h,h,h,h)^T$ and thus by (\ref{trans}) the polynomial $F_A$ satisfies
\begin{equation}\label{homog}
F_A(\lambda^{q_0}x_0,\lambda^{q_1}x_1,\lambda^{q_2}x_2,\lambda^{q_3}x_3) = \lambda^h F_A(x_0,x_1,x_2,x_3)
\end{equation}
for all $\lambda\in k^\times$.  This says that $F_A$ is \emph{quasihomogeneous} of degree $h$ and weight system ${\bf q}=(q_0,q_1,q_2,q_3)$.

\begin{defn}\label{CY-defn}
We will say that the weighted Delsarte matrix $A$ satisfies the \emph{Calabi-Yau requirement} if (with notation as above) we have $h=\sum_{i=0}^3 q_i$.
\end{defn}

\begin{remark}\label{CY-remark}
We note by (\ref{eq1}) that $A$ satisfies the Calabi-Yau requirement if and only if the entries of $A^{-1}$ sum to 1.
\end{remark}

For the rest of this section, we assume that we are working with a Delsarte matrix $A$ that also satisfies the Calabi-Yau requirement. 

Next consider the following subgroup of $(k^\times)^4$:
\[
\Aut(F_A) = \set{(\lambda_0,\lambda_1,\lambda_2,\lambda_3)\in (k^\times)^4 \ | \ F(\lambda_0x_0,\lambda_1x_1,\lambda_2x_2,\lambda_3x_3) = F_A(x_0,x_1,x_2,x_3)}.
\]
One may use (\ref{A-det}) to show that $\Aut(F_A)$ is a subgroup of $\mu_d^4$, where $d$ is the smallest positive integer such that the matrix
\begin{equation}\label{B-d-defn}
B = d A^{-1}
\end{equation}
has integer entries and $\mu_d$ denotes the group of $d$th roots of unity in $k$.  Moreover, one may show that $\#\Aut(F_A) = \abs{\det(A)}$.

\begin{lem}
We have $h\mid d$, $d\mid {\det(A)}$, and $\det(A)\mid d^4$.  In particular, the primes dividing $d$ and $\det(A)$ are the same.
\end{lem}

\begin{proof}
One may characterize $h$ as the positive generator of the following ideal in $\Z$:
\[
\set{m\in\Z \ | \ mA^{-1}(1,1,1,1)^T \in\Z^4}.
\]
By its definition in (\ref{B-d-defn}), $d$ belongs to this ideal and hence we conclude $h\mid d$.  Likewise, we may characterize $d$ as the positive generator of the ideal
\[
\set{m\in\Z \ | \ mA^{-1} \in M_{4\times 4}(\Z)},
\]
and as $\det(A)A^{-1}$ is the adjugate matrix of $A$, we must have $d\mid \det(A)$.  Finally, as $\Aut(F_A)$ has order $\abs{\det(A)}$, it must divide the order of the group $\mu_d^4$.  The concluding remark in the statement of the lemma follows from the fact that $d\mid \det(A)$ and $\det(A)\mid d^4$.
\end{proof}

  If we fix a primitive $d$th root of unity $\zeta\in k$ and select some ${\bf a}=(a_0,a_1,a_2,a_3)\in (\Z/d)^4$, then by (\ref{trans}) we find that
\[
(\zeta^{a_0},\zeta^{a_1},\zeta^{a_2},\zeta^{a_3}) \in \Aut(F_A) \quad \iff \quad A{\bf a}^T = 0^T \text{ in } \Z/d.
\]
Due to this, we will typically (but not always) use this additive notation:
\[
(\zeta^{a_0},\zeta^{a_1},\zeta^{a_2},\zeta^{a_3})\in \mu_d^4 \quad \longleftrightarrow \quad {\bf a}=(a_0,a_1,a_2,a_3)\in (\Z/d)^4:
\]
To reduce ambiguity between these notations, we will also use greek variables for the multiplicative notation and roman variables for the additive notation.

We now identify two subgroups of $\Aut(F_A)$.  First (in our additive notation) let
\[
{\bf j}_{A} = \frac{d}{h}{\bf q} = \left(\frac{q_0d}{h},\frac{q_1d}{h},\frac{q_2d}{h},\frac{q_3d}{h}\right).
\]
By (\ref{homog}), we see ${\bf j}_A\in \Aut(F_A)$, and we denote by $J_{F_A} = \inn{{\bf j}_A}$ the cyclic subgroup that it generates.  Second, define
\[
\SL(F_A) = \Big\{{\bf a}\in \Aut(F_A) \ \Big\vert \ \sum_{i=0}^3 a_i = 0\Big\}.
\]
(In the multiplicative notation, an element $(\lambda_0,\lambda_1,\lambda_2,\lambda_3)\in \Aut(F_A)$ belongs to $\SL(F_A)$ precisely when $\prod_{i=0}^3 \lambda_i = 1$; by viewing $(\lambda_0,\lambda_1,\lambda_2,\lambda_3)$ as a diagonal matrix, one better comprehends this widely-used notation for this subgroup of $\Aut(F_A)$.)  The Calabi-Yau condition on $A$ implies that $J_{F_A} \subseteq \SL(F_A)$.

\begin{defn}
A \emph{BHK pair} will be a pair $(A,G)$ such that:
\begin{enumerate}
\item $A$ is a weighted Delsarte matrix satisfying the Calabi-Yau condition.
\item $G$ is a subgroup satisfying $J_{F_A} \subseteq G\subseteq \SL(F_A)$.
\end{enumerate}
\end{defn}

\section{BHK surfaces of K3 type}\label{K3-sec}

From these algebraic ideas, we now pass to a more geometric setting, by using a BHK pair to form a hypersurface in weighted projective space.  A standard reference for weighted projective spaces is \cite{Dolg-WPS}.

Recall that $k$ is a field of characteristic $p\geq 0$.  We let $(A,G)$ denote a BHK pair and assume the following further conditions:
\begin{itemize}
\item The weight system ${\bf q} = (q_0,q_1,q_2,q_3)$ satisfies $\gcd(q_i,q_j,q_\ell) = 1$ whenever $0\leq i<j<\ell\leq 3$.
\item $p\nmid q_i$ for all $i$.
\item $p\nmid d$.
\end{itemize}
The polynomial $F_A$ defines a hypersurface $X_A$ (called a \emph{weighted Delsarte surface}) in the weighted projective space $\PP^3({\bf q}) = \PP^3(q_0,q_1,q_2,q_3)$ of degree $h$, and $X_A$ has at most cyclic quotient singularities.  We assume that $X_A$ is:
\begin{itemize}
\item \emph{quasi-smooth}, meaning that the affine cone over $X_A$ has the origin as its only singular point, and
\item \emph{well-formed}, meaning that the intersection of the (one-dimensional) singular locus of $\PP^3({\bf q})$ with $X_A$ is zero-dimensional or empty.
\end{itemize}
Due to this, and the fact that $A$ satisfies the Calabi-Yau requirement, the minimal resolution $\tilde X_A$ will be a K3 surface (see \cite{Goto-symp}).

The group $\Aut(F_A)$ acts upon points of $X_A$ via
\[
(\lambda_0,\lambda_1,\lambda_2,\lambda_3)\cdot(x_0:x_1:x_2:x_3) = (\lambda_0x_0:\lambda_1x_1:\lambda_2x_2:\lambda_3x_3),
\]
and this extends to an action upon $\tilde X_A$.  Note that the action of the subgroup $J_{F_A}$ is trivial.  Moreover one may show (see \cite[Prop.\ 1]{ABS}) that $\SL(F_A)$ is precisely the subgroup of elements in $\Aut(F_A)$ that act \emph{symplectically} upon $\tilde X_A$, i.e., that act trivially on the canonical bundle of $\tilde X_A$.  Hence the group $G$, which satisfies $J_{F_A}\subseteq G\subseteq \SL(F_A)$, also acts symplectically on $\tilde X_A$.  Since $p\nmid d$, the minimal resolution $Z_{A,G}$ of the quotient surface $\tilde X_A/G$ will again be a K3 surface; see \cite{Goto-symp} for more details.

Hence, under certain conditions, a BHK pair $(A,G)$ gives rise to the corresponding geometric object $Z_{A,G}$.  We summarize all of this with the following:

\begin{defn}
Let $(A,G)$ be a BHK pair such that $\gcd(q_i,q_j,q_\ell)=1$ for any three distinct weights in ${\bf q}$, $p\nmid q_i$ for all $i$, and $p\nmid d$.  If the hypersurface $X_A$ in $\PP^3({\bf q})$ is quasi-smooth and well-formed, we will call $(A,G)$ an \emph{adequate BHK pair}.  Furthermore, if $Z_{A,G}$ denotes the K3 surface obtained as the minimal resolution of $\tilde X_A/G$, then $Z_{A,G}$ is called a \emph{K3 surface of BHK type}.
\end{defn}

\begin{lem}\label{lifting-lemma}
Suppose that the BHK pair $(A,G)$ is an adequate BHK pair over a field of characteristic $p>0$, i.e., that it gives rise to a K3 surface $Z_{A,G}$ in characteristic $p>0$.  Then the same is true over a field of characteristic 0.
\end{lem}

\begin{proof}
Put briefly, this follows from the fact that ``well-formed'' and ``quasi-smooth'' are open conditions.  To see this, notice that $F_A\in \Z[x_0,x_1,x_2,x_3]$ cuts out a codimension one subscheme $\mathfrak X_A' \subseteq \Spec \Z[x_0,x_1,x_2,x_3]$.  Likewise, if $\Z[x_0,x_1,x_2,x_3]$ is graded so that each $x_i$ has weight $q_i$, then $F_A$ cuts out the codimension one subscheme $\mathfrak X_A \subseteq \PP_\Z^3({\bf q}) = \Proj \Z[x_0,x_1,x_2,x_3]$.  Note that $\mathfrak X_A'$ is the affine cone over $\mathfrak X_A$.

Now assume that $(A,G)$ defines the K3 surface $Z_{A,G}$ over a field of characteristic $p>0$.  
Let $\mathfrak Y\subseteq \mathfrak X_A'$ denote the singular locus of $\mathfrak X_A'$.  The fact that $Z_{A,G}$ is quasi-smooth means that the fiber of $\mathfrak Y\to\Spec \Z$ over $p$ is zero-dimensional.  Hence by semi-continuity, the same is true about the generic fiber; in other words, $(A,G)$ defines a quasismooth surface in characteristic 0.

Likewise, if $\mathfrak Z\subseteq \mathfrak X_A$ denotes the intersection of $\mathfrak X_A$ with the singular locus of $\PP^3_\Z({\bf q})$, the hypothesis of the lemma implies that the fiber of $\mathfrak Z\to\Spec \Z$ over $p$ is either zero-dimensional or empty, and so the same must be true of the generic fiber.  Hence $(A,G)$ defines a well-formed surface in characteristic zero as well.
\end{proof}

\begin{remark}\label{KS-remark} While they will not play a significant role in the proof of Theorem \ref{main-thm}, there are two other facets of an adequate BHK pair $(A,G)$ worth noting here.

First, in \cite[Theorem 1]{KS}, Kreuzer and Skarke classify those weighted Delsarte matrices $A$ such that the corresponding hypersurface $X_A$ will be quasismooth.  (Note while the result in \cite{KS} is stated for characteristic zero, the proof of Lemma \ref{lifting-lemma} shows how their results hold in all characteristics.)  Their classification shows $X_A$ is quasismooth if and only if the polynomial $F_A$ is a direct sum of so-called \emph{atomic types}, of which there are three:

\medskip

\begin{tabular}{ll}
{Fermat}: & $y^e$ \\
{Loop}: & $y_0^{e_0}y_1+y_1^{e_1}y_2+\cdots+y_k^{e_k}y_0$ \\
Chain: & $y_0^{e_0}y_1+y_1^{e_1}y_2+\cdots+y_{k-1}^{e_{k-1}}y_k+y_k^{e_k}$
\end{tabular}

\medskip

The second point is that, given the assumptions made about the weight system ${\bf q}=(q_0,q_1,q_2,q_3)$ at the beginning of this section, ${\bf q}$ will belong to one of 95 possible weight systems (counted up to reordering the $q_i$); see \cite{Reid, Yonemura}.  Taken together with the first point, this significantly restricts the possibilities for the matrix $A$ in an adequate BHK pair.  (On the other hand, their number is still quite large, as demonstrated by the partial lists of examples in \cite{ABS, CLPS, CP}.)
\end{remark}

\section{BHK mirrors}\label{mirror-sec}

Let $(A,G)$ be an adequate BHK pair.  Since $X_A$ is then quasismooth, the classification result of Kreuzer and Skarke mentioned in Remark \ref{KS-remark} implies that its transpose $A^T$ is also a weighted Delsarte matrix. Furthermore, Remark \ref{CY-remark} shows that $A^T$ will also satisfies the Calabi-Yau requirement.  Just as before, one obtains from $A^T$ a polynomial $F_{A^T}$, which will be quasihomogeneous of some degree $h_T$ and weight system ${\bf q}_T$.  (We note that, in general, the degree and weight system of $F_{A^T}$ will differ from that of $F_A$.)  Just as with $F_A$, one may define the groups
\[
J_{F_{A^T}} \subseteq \SL(F_{A^T}) \subseteq \Aut(F_{A^T}).
\]
The group $\Aut(F_{A^T})$ is a subgroup of $\mu_d^4$ (for the same value of $d$ as in \S\ref{const-sec}) and again we have $\#\Aut(F_{A^T}) = \abs{\det(A)}$.  Alternatively, in the additive notation, we may view $\Aut(F_{A^T})$ as the following subgroup of $(\Z/d)^4$:
\[
\Aut(F_{A^T}) = \set{{\bf a}\in (\Z/d)^4 \ | \ {\bf a}A= (A^T{\bf a}^T)^T = 0} 
\]

Let us now use the additive notation to define a pairing
\[
\inn{ \cdot,\cdot}: \Aut(F_{A^T})\times \Aut(F_A) \to \Z/d^2.
\]
This pairing is defined by the rule
\[
\inn{{\bf a},{\bf b}} = \tilde{\bf a}A\tilde{\bf b}^T,
\]
where $\tilde{\bf a}, \tilde{\bf b}$ are any lifts of ${\bf a}, {\bf b}$ to the group $(\Z/d^2)^4$; one may verify that $\inn{{\bf a},{\bf b}}$ is independent of the choices of these lifts, by using the fact that ${\bf a}A = 0$ and $A{\bf b}^T=0$ in $(\Z/d)^4$.  This pairing is $\Z/d^2$-bilinear, if one regards $g\in\Z/d^2$ as acting upon each of the factors $\Aut(F_A)$ and $\Aut(F_{A^T})$ via scalar multiplication after applying the reduction map $\Z/d^2\to \Z/d$.

From the adequate BHK pair $(A,G)$ above, one may now define the \emph{dual group} $G^T$ as
\[
G^T = \set{{\bf a}\in \Aut(F_{A^T}) \ | \ \inn{{\bf a},{\bf b}} = 0 \text{ for all } {\bf b}\in G}.
\]
One may check that this is equivalent to the definition of $G^T$ in \cite{Kelly}.  (Also see \cite[\S3]{ABS} for a discussion of several equivalent definitions of $G^T$.)  Some useful facts about this construction are that $(G^T)^T=G$, $\Aut(F_A)^T=\set{0}$, $J_{F_A}^T=\SL(F_{A^T})$, and $J_{F_{A^T}}\subseteq G^T \subseteq \SL(F_{A^T})$.  (Note that the latter two facts require the Calabi-Yau assumption and the fact that $J\subseteq G\subseteq \SL(F_A)$.)

From the discussion above, the pair $(A^T, G^T)$ will automatically be a BHK pair, and the surface $X_{A^T}$ will be quasismooth.  Hence $(A^T,G^T)$ will be an adequate BHK pair if and only if $p\nmid q_{i,T}$ (where ${\bf q}_T = (q_{0,T},q_{1,T},q_{2,T},q_{3,T})$) and $X_{A^T}$ is also well-formed.%

\begin{defn}
Let $(A,G)$ be an adequate BHK pair.  If the BHK pair $(A^T,G^T)$ is also an adequate BHK pair, then we will call it the \emph{BHK mirror pair} of $(A,G)$.  Furthermore, we will say that the K3 surface $Z_{A^T,G^T}$ is the \emph{BHK mirror surface} of the surface $Z_{A,G}$.
\end{defn}

\section{Picard numbers of K3 surfaces of BHK type}\label{Pic-no-sec}

Let $(A,G)$ be a BHK pair.  We will define several subsets of the additive group $(\Z/d)^4$:
\begin{eqnarray}
M_d &=& \set{{\bf a} = (a_0,a_1,a_2,a_3)\in(\Z/d)^4 \ \Big\vert \ \sum_{i=0}^3 a_i = 0} \label{Md-defn} \\
\mathfrak A_d &=& \set{{\bf a}\in M_d \ | \ a_i \neq 0 \text{ for all } 0\leq i\leq 3} \label{Ad-defn} 
\end{eqnarray}
Then $M_d$ is a subgroup of $(\Z/d)^4$, while $\mathfrak A_d$ is merely a subset that is closed under scalar multiplication by $(\Z/d)^\times$. We define the following function upon $\mathfrak A_d$:

\begin{defn}
The \emph{age} of ${\bf a}=(a_0,a_1,a_2,a_3)\in \mathfrak A_d$ is the integer
\[
S({\bf a}) = \sum_{i=0}^3 \inn{\frac{\hat a_i}{d}},
\]
where $\hat a_i\in \Z$ is any lift of $a_i$ from $\Z/d$ to $\Z$, and $\inn{x} = x-\lfloor x \rfloor$ denotes the fractional part of $x\in\R$.
\end{defn}

\begin{remark}
In the literature surrounding BHK mirror symmetry, the concept of {age} is typically defined for any element of $\Aut(F_A)$.  The notion is also useful in other contexts (for instance, see \cite{IR}).  For our purposes, we will only need to consider it for elements in $\mathfrak A_d$.
\end{remark}

The following properties of $S({\bf a})$ are straightforward to verify:

\begin{prop}\label{neg-prop}
For ${\bf a}\in \mathfrak A_d$ we have:
\begin{enumerate}
\item $S({\bf a}) \in\set{1,2,3}$
\item $S(-{\bf a}) = 4-S({\bf a})$
\end{enumerate}
\end{prop}

Now define
\[
\mathfrak B_d(0) = \set{{\bf a}\in \mathfrak A_d \ | \ S(t{\bf a}) = 2 \text{ for all } t\in (\Z/d)^\times}.
\]
If $p>0$ is a prime such that $p\nmid d$, let $f$ denote the order of the (congruence class of) $p$ in the multiplicative group $(\Z/d)^\times$.  We also define
\[
\mathfrak B_d(p) = \Big\{{\bf a}\in \mathfrak A_d \ \Big\vert \ \sum_{j=0}^{f-1} S(tp^j{\bf a}) = 2 \text{ for all } t\in (\Z/d)^\times\Big\}.
\]
Finally, for $p=0$ or a prime $p\nmid d$, let
\[
\mathfrak I_d(p) = \mathfrak A_d\setminus \mathfrak B_d(p).
\]

The main theorem of \cite{Kelly} is:

\begin{thm}[Kelly]\label{Kelly-thm}
Let $(A,G)$ and $(A^T,G^T)$ be BHK mirror pairs.  Over a field of characteristic $p=0$ or $p\nmid d$, the geometric Picard numbers of the BHK mirror surfaces $Z_{A,G}$ and $Z_{A^T,G^T}$ are given by
\begin{eqnarray*}
\rho(Z_{A,G}) &=& 22-\#(\mathfrak I_d(p)\cap G^T) \\
\rho(Z_{A^T,G^T}) &=& 22-\#(\mathfrak I_d(p)\cap G)
\end{eqnarray*}
\end{thm}

\section{Quotients of Fermat surfaces}

In this section, all varieties will be defined over a field of characteristic 0.  We will let
\[
Y_d = V(X_0^d+X_1^d+X_2^d+X_3^d)
\]
denote the Fermat surface of degree $d$ in $\PP^3$.

The group $\mu_d^4$ acts upon $Y_d$ by multiplication of each coordinate:
\[
(\lambda_0,\lambda_1,\lambda_2,\lambda_3)\cdot(X_0:X_1:X_2:X_3) = (\lambda_0X_0:\lambda_1X_1:\lambda_2X_2:\lambda_3X_3).
\]
Letting $\Delta\hookrightarrow\mu_d^4$ denote the diagonal subgroup, we obtain in this way an injection $\mu_d^4/\Delta \hookrightarrow \Aut(Y_d)$.  By functoriality we obtain a representation of $\mu_d^4/\Delta$ upon the singular cohomology group $H^2(Y_d,\Q)$.

We note that the group $\Hom(\mu_d^4/\Delta,\C^\times)$ of characters of $\mu_d^4/\Delta$ may be identified with the group $M_d$ defined in (\ref{Md-defn}).  In particular,
\[
{\bf a} = (a_0,a_1,a_2,a_3)\in M_d \quad \implies \quad {\bf a}\big(\big(\lambda_0,\lambda_1,\lambda_2,\lambda_3)\Delta\big) = \prod_i \lambda_i^{a_i}.
\]
Hence upon extending scalars to $\C$, we obtain a decomposition of the middle singular cohomology of $Y_d$:
\[
H^2(Y_d,\C) = \bigoplus_{{\bf a}\in M_d} V({\bf a}),
\]
where $\mu_d^4/\Delta$ acts via the character ${\bf a}$ on $V({\bf a})$.  Recalling the subset $\mathfrak A_d\subseteq M_d$ in (\ref{Ad-defn}), one can show (see \cite{Shioda-Fermat}, for instance) that
\[
\dim_{\C} V({\bf a}) = 
\begin{cases}
1 & \text{ if } {\bf a}=0 \text{ or } {\bf a}\in\mathfrak A_d \\
0 & \text{ otherwise},
\end{cases}
\]
so that $H^2(Y_d,\C)$ becomes the following direct sum of one-dimensional spaces:
\[
H^2(Y_d,\C) = V(0)\oplus \bigoplus_{{\bf a}\in \mathfrak A_d} V({\bf a}).
\]
Moreover, as this decomposition arises from a group of automorphisms of $Y_d$, each of the subspaces $V({\bf a})$ is a Hodge structure (over $\C$) and one can ask about its Hodge type.  

\begin{prop}\label{S-Hodge-prop}
If ${\bf a}\in \mathfrak A_d$, then $V({\bf a})\subseteq H^{2-q,q}(Y_d,\C)$ if and only if $S({\bf a}) = q+1$.  Furthermore $V(0)\subseteq H^{1,1}(Y_d,\C)$.
\end{prop}

\begin{proof}
See \cite{Shioda-Fermat}.
\end{proof}

Now let $(A,G)$ be an adequate BHK pair.  In \cite[Cor.\ 3.2]{Kelly2}, the K3 surface $Z_{A,G}$ is shown to be birational to the quotient $Y_d/\Gamma$, for a particular subgroup $\Gamma \subseteq \mu_d^4$ (whose precise definition will not be needed here).  Thus we have an isomorphism of Hodge structures 
\[
H^2(Z_{A,G},\Q) \simeq {\bf E}\oplus H^2(Y_d,\Q)^{\Gamma},
\]
where ${\bf E}$ is generated by the classes of some $(-1)$-curves on $Z_{A,G}$.  Tensoring with $\C$, we obtain
\begin{eqnarray*}
H^2(Z_{A,G},\C) &\simeq& {\bf E_\C}\oplus H^2(Y_d,\C)^{\Gamma} \\
&=& {\bf E}_\C \oplus \left( V(0)\oplus \bigoplus_{{\bf a}\in \mathfrak A_d} V({\bf a}) \right)^\Gamma \\
&=& {\bf E}_\C \oplus V(0)\oplus \bigoplus_{{\bf a}\in \mathfrak A_d\cap L(\Gamma)} V({\bf a}),
\end{eqnarray*}
where
\[
L(\Gamma) = \set{{\bf a}\in M_d \ | \ {\bf a}(\gamma)=1 \text{ for all } \gamma\in \Gamma}.
\]
But in \cite[Prop.\ 3]{Kelly}, Kelly shows that
\[
L(\Gamma) = G^T \cap M_d,
\]
where $G^T$ is the dual group of the group $G$.  Since $\mathfrak A_d\subseteq M_d$, we arrive at
\begin{equation}\label{G-dual-decomp}
H^2(Z_{A,G},\C) \simeq {\bf E}_\C \oplus V(0)\oplus \bigoplus_{{\bf a}\in \mathfrak A_d\cap G^T} V({\bf a}).
\end{equation}

\begin{prop}\label{exactly-one-prop}
Let $(A,G)$ and $(A^T,G^T)$ be BHK mirror pairs.  Each of the sets $\mathfrak A_d\cap G$ and $\mathfrak A_d\cap G^T$ contains exactly one element of age 1.
\end{prop}

\begin{proof}
By looking at the $(2,0)$-part of the Hodge structures in (\ref{G-dual-decomp}) and applying Proposition \ref{S-Hodge-prop}, we have
\[
H^{2,0}(Z_{A,G},\C) \simeq \bigoplus_{\substack{{\bf a}\in \mathfrak A_d\cap G^T \\ S({\bf a}) = 1}} V({\bf a}).
\]
But $Z_{A,G}$ is a K3 surface, so the dimension of the left side must be 1.  Hence there is exactly one ${\bf a}\in \mathfrak A_d\cap G^T$ such that $S({\bf a})=1$.  By symmetric reasoning we have
\[
H^{2,0}(Z_{A^T,G^T},\C) \simeq \bigoplus_{\substack{  {\bf a}\in \mathfrak A_d\cap G \\ S({\bf a}) = 1 }} V({\bf a}),
\]
from which we conclude that $\mathfrak A_d\cap G$ also has exactly one element of age 1.
\end{proof}

\section{Description in terms of orbits}\label{orbit-sec}

The results in this section are of a strictly combinatorially nature.  Specifically, we will work with an integer $d\geq 1$, a subgroup $H\subseteq (\Z/d)^4$, and primes $p$ that do not divide $d$.  Our goal is to analyze the structure of the sets $\mathfrak I_d(0)\cap H$ and $\mathfrak I_d(p)\cap H$ under an assumption on $H$ inspired by Proposition \ref{exactly-one-prop}.  The results obtained will be applied to the geometric setting in the next section.

We let $U_d = (\Z/d)^\times$.  This group acts upon $(\Z/d)^4$ by multiplication, and it preserves the subset $\mathfrak A_d$ and any subgroup of $(\Z/d)^4$.  In this section, we fix some subgroup $H\subseteq (\Z/d)^4$.  Hence $U_d$ acts upon the set $\mathfrak A_d\cap H$, and we let
\[
\Phi_H = (\mathfrak A_d\cap H)/U_d
\]
denote the collection of orbits of this action.  For a prime $p\nmid d$, let $\inn{p}\subseteq U_d$ denote the subgroup generated by (the congruence class of) $p$.  Since the subgroup $\inn{p}$ will of course preserve any orbit $\mathcal O\in \Phi_H$ of the larger group $U_d$, we make another piece of notation: Let
\[
\Phi_H(p,\mathcal O) = \mathcal O/\inn{p}
\]
denote the collection of orbits of $\inn{p}$ acting upon a $U_d$-orbit $\mathcal O$.  Thus we have decomposed each $\mathcal O\in \Phi_H$ as the disjoint union
\[
\mathcal O = \coprod_{T\in \Phi_H(p,\mathcal O)} T
\]
and decomposed $\mathfrak A_d\cap H$ as
\[
\mathfrak A_d\cap H = \coprod_{\mathcal O\in \Phi_H} \mathcal O = \coprod_{\mathcal O\in \Phi_H} \left( \coprod_{T\in \Phi_H(p,\mathcal O)} T \right).
\]

We now restate the definitions of the sets $\mathfrak B_d(0)\cap H$ and $\mathfrak I_d(0)\cap H$ in terms of the orbits $\Phi_H$.  First note that if ${\bf a}\in \mathfrak A_d\cap H$, then
\begin{eqnarray*}
{\bf a}\in \mathfrak B_d(0)\cap H &\iff&  S(t{\bf a}) = 2 \text{ for all } t\in U_d \\
&\iff& S({\bf b}) = 2 \text{ for all } {\bf b} \in U_d\cdot {\bf a},
\end{eqnarray*}
where $U_d\cdot {\bf a}$ denotes the $U_d$-orbit of ${\bf a}$.  Therefore $\mathfrak B_d(0)\cap H$ is a union of certain orbits in $\Phi$, namely those $\mathcal O\in \Phi_H$ that only contain elements of age 2.  In symbols,
\[
\mathfrak B_d(0)\cap H = \coprod \set{\mathcal O \in \Phi_H \ | \ S({\bf a}) = 2 \text{ for all } {\bf a}\in \mathcal O}.
\]
Therefore, taking the complement of $\mathfrak B_d(0)\cap H$ inside $\mathfrak A_d\cap H$, we obtain
\[
\mathfrak I_d(0)\cap H = \coprod \set{\mathcal O \in \Phi_H \ | \ S({\bf a}) \neq 2 \text{ for some } {\bf a}\in \mathcal O}.
\]
But if $S({\bf a})\neq 2$, then by Proposition \ref{neg-prop}, either $S({\bf a})=1$ or $S(-{\bf a})=1$.  Therefore, we may alternatively write
\begin{equation}\label{I-0-redefined}
\mathfrak I_d(0)\cap H = \coprod \set{\mathcal O \in \Phi_H \ | \ S({\bf a}) = 1 \text{ for some } {\bf a}\in \mathcal O}.
\end{equation}
In words, this says that $\mathfrak I_d(0)\cap H$ is the union of those $U_d$-orbits $\mathcal O$ that contain at least one element of age 1.

From (\ref{I-0-redefined}), we can deduce:

\begin{prop}\label{0-prop}
Suppose that $\mathfrak A_d\cap H$ contains exactly one element ${\bf a}_0$ of age 1.  If $h$ denotes the order of ${\bf a}_0$ in $(\Z/d)^4$, then
\[
\mathfrak I_d(0)\cap H = \set{n{\bf a}_0 \ | \ 1\leq n\leq h, (n,h)=1}
\]
is a set of cardinality $\phi(h)$, where $\phi$ denotes Euler's totient function.
\end{prop}

\begin{proof}
By (\ref{I-0-redefined}), $\mathfrak I_d(0)\cap H$ is equal to the orbit of ${\bf a}_0$ under $U_d$.  That is, $\mathfrak I_d(0)\cap H = R_d$ where
\[
R_d = \set{m{\bf a}_0 \ | \ 1\leq m\leq d, (m,d)=1}.
\]
Letting $R_h=\set{n{\bf a}_0 \ | \ 1\leq n\leq h, (n,h)=1}$, we must show that $R_d=R_h$.

If $m{\bf a}_0\in R_d$, let $m\equiv m_0\pmod{h}$, where $1\leq m_0\leq h$.  Since $h\mid d$, we have $(m_0,h)=1$, and so $m{\bf a}_0=m_0{\bf a}_{0} \in R_h$.  Thus $R_d\subseteq R_h$.

For the opposite inclusion, take $n{\bf a}_0\in R_h$.  The main obstacle to concluding that $n{\bf a}_0\in R_d$ is that we might have $(n,h)=1$ but $(n,d)>1$.  If the primes dividing $h$ and $d$ are identical, this will not happen; so assuming otherwise, suppose that $Q>1$ denotes the product of all primes that divide $d$ but not $h$.  By the Chinese Remainder Theorem, there exists an integer $N$ satisfying $1\leq N\leq hQ \leq d$, $N\equiv n \pmod{h}$ and $N\equiv 1\pmod{Q}$.  Then $N$ will be not be divisible by any primes dividing $d$, and thus $n{\bf a}_0 = N{\bf a}_0\in R_d$.

Finally, we note that all elements of $R_h$ are distinct since the order of ${\bf a}_0$ is $h$, so it has cardinality $\phi(h)$.
\end{proof}

Likewise, we may recast $\mathfrak B_d(p)\cap H$ and $\mathfrak I_d(p)\cap H$ in terms of the orbits $\Phi_H(p)$.  Recalling that $f$ is the order of $p$ in $U_d$, we note that if ${\bf a}\in \mathfrak A_d\cap H$ then
\begin{eqnarray*}
{\bf a}\in \mathfrak B_d(p)\cap H &\iff& \sum_{j=0}^{f-1} S(tp^j{\bf a}) = 2f \text{ for all } t\in U_d \\
&\iff& \sum_{j=0}^{f-1} S(p^j{\bf b}) = 2f \text{ for all } {\bf b} \in U_d\cdot {\bf a}.
\end{eqnarray*}
Now the sum $\sum_{j=0}^{f-1} S(p^j{\bf b})$ consists of $f$ terms, so we have
\[
\sum_{j=0}^{f-1} S(p^j{\bf b}) = 2f \iff \sum_{j=0}^{f-1} (S(p^j{\bf b})-2)=0.
\]
The quantity $S(p^j{\bf b})-2$ takes values in $\set{-1,0,1}$ by Proposition \ref{neg-prop}, and so the equation on the right side is true if and only if there are an equal number of occurrences of 1 and $-1$ in the sum.  That is, we have $\sum_{j=0}^{f-1} S(p^j{\bf b}) = 2f$ if and only if 1 and 3 appear in equal numbers in the sequence
\[
S({\bf b}), S(p {\bf b}), \ldots, S(p^{f-1}{\bf b}).
\]
Finally, we note that sequence 
\begin{equation}\label{sequence}
{\bf b}, p {\bf b}, \ldots, p^{f-1}{\bf b}
\end{equation}
is simply the orbit of ${\bf b}$ under $\inn{p}$, but possibly repeated a certain number of times.  (More precisely, if the $\inn{p}$-orbit of ${\bf b}$ has size $f'$, then $f'\mid f$ and the sequence (\ref{sequence}) repeats this orbit $f/f'$ times.)  Hence we conclude
\begin{eqnarray*}
{\bf a}\in \mathfrak B_d(p)\cap H &\iff& \sum_{j=0}^{f-1} S(p^j{\bf b}) = 2f \text{ for all } {\bf b} \in U_d\cdot {\bf a} \\
&\iff& \#\big((\inn{p}\cdot {\bf b}) \cap S^{-1}(1)\big) = \#\big((\inn{p}\cdot {\bf b}) \cap S^{-1}(3)\big) \text{ for all } {\bf b} \in U_d\cdot {\bf a}
\end{eqnarray*}
As with $\mathfrak B_d(0)\cap H$, we see that $\mathfrak B_d(p)\cap H$ is a union of certain orbits in $\Phi_H$.  Looking at the preceding equivalence, we can state this precisely as:
\[
\mathfrak B_d(p)\cap H = \coprod \set{\mathcal O\in \Phi_H \ | \ \#(T\cap S^{-1}(1)) = \#(T\cap S^{-1}(3)) \text{ for all } T\in \Phi_H(p,\mathcal O)}.
\]
Therefore we have
\begin{equation}\label{I-p-redefined}
\mathfrak I_d(p)\cap H = \coprod \set{\mathcal O\in \Phi_H \ | \ \#(T\cap S^{-1}(1)) \neq \#(T\cap S^{-1}(3)) \text{ for some } T\in \Phi_H(p,\mathcal O)}.
\end{equation}
In words, $\mathfrak I_d(p)\cap H$ is the union of those $U_d$-orbits $\mathcal O$ that contain at least one $\inn{p}$-orbit $T$ having unequal numbers of elements of age 1 and 3.

\begin{prop}\label{p-prop}
Suppose that $\mathfrak A_d\cap H$ contains exactly one element ${\bf a}_0$ of age 1. If $h$ denotes the order of ${\bf a}_0$ in $(\Z/d)^4$, then we have the following two possibilities for $\mathfrak I_d(p)\cap H$:
\begin{enumerate}
\item If $p^\ell \equiv -1\pmod{h}$ for some $\ell$, then $\mathfrak I_d(p)\cap H = \emptyset$.
\item If $p^\ell\not\equiv -1 \pmod{h}$ for all $\ell$, then $\mathfrak I_d(p)\cap H = \mathfrak I_d(0)\cap H$.
\end{enumerate}
\end{prop}

\begin{proof}
By Proposition \ref{neg-prop}, the hypothesis implies that $-{\bf a}_0$ is the only element of $\mathfrak A_d\cap H$ of age 3.  Hence $U_d\cdot {\bf a}_0$ is the only orbit in $\Phi_H$ that doesn't consist entirely of elements of age 2.  So by (\ref{I-p-redefined}), $\mathfrak I_d(p)\cap H$ can only be empty or equal to $U_d \cdot {\bf a}_0$, which in turn is equal to $\mathfrak I_d(0)\cap H$ by Proposition \ref{0-prop}.

If we assume that $p^\ell\equiv -1 \pmod{h}$, we find that $\inn{p}\cdot{\bf a}_0$ contains one element of age 1 and one element of age 3, and this gives the first statement.  But assuming $p^\ell\not \equiv -1 \pmod{h}$ for all $k$, it follows that the $\inn{p}\cdot {\bf a}_0$ contains one element of age 1 and no elements of age 3; hence $U_d\cdot {\bf a}_0$ is contained in $\mathfrak I_d(p)\cap H$, giving the second statement.
\end{proof}

\section{Picard numbers of K3 surfaces of BHK type} 

Let $(A,G)$ be an adequate BHK pair, and let $d$ be defined as in (\ref{B-d-defn}).  As before, this defines a polynomial $F_A$ that is quasihomogeneous of degree $h$ with weight system ${\bf q}=(q_0,q_1,q_2,q_3)$.  By assumption, the group $G$ satisfies $J_{F_A}\subseteq G\subseteq \SL(F_A)$, and thus $G$ contains the element
\[
{\bf j}_{A} = \frac{d}{h}{\bf q} = \left(\frac{q_0d}{h},\frac{q_1d}{h},\frac{q_2d}{h},\frac{q_3d}{h}\right),
\]
which is a generator of $J_{F_A}$.  Now
\begin{equation}\label{S-of-j}
S({\bf j}_{A}) = \sum_{i=0}^3 \inn{\frac{q_id/h}{d}} = \sum_{i=0}^3 \inn{\frac{q_i}{h}} = \sum_{i=0}^3 \frac{q_i}{h} = 1
\end{equation}
due to the Calabi-Yau requirement $\sum_i q_i = h$.

\begin{lem}\label{exactly-one-lemma}
Let $(A,G)$ and $(A^T,G^T)$ be BHK mirror pairs.  Then ${\bf j}_A$ is the only element of age 1 in the set $\mathfrak A_d\cap G^T$, where $G^T$ is the dual group of $G$.
\end{lem}

\begin{proof}
By Lemma \ref{lifting-lemma}, we know that $(A,G)$ gives rise to the K3 surface in characteristic 0 that we have denoted by $Z_{A,G}$.  Applying (\ref{S-of-j}) and  Proposition \ref{exactly-one-prop} to this surface, we obtain the desired conclusion.
\end{proof}

Our main result now follows:

\begin{thm}\label{main-thm}
Let $(A,G)$ and $(A^T,G^T)$ be BHK mirror pairs.  Let $h$ (resp.\ $h_T$) denote the degree of the quasihomogeneous polynomial $F_A$ (resp.\ $F_{A^T}$).  Over the field $k$, the (geometric) Picard numbers of the BHK mirror surfaces $Z_{A,G}$ and $Z_{A^T,G^T}$ are given as follows:
\begin{enumerate}
\item If $\chr k = 0$ then
\begin{eqnarray*}
\rho(Z_{A,G}) &=& 22-\phi(h_T) \\
\rho(Z_{A^T,G^T}) &=& 22-\phi(h).
\end{eqnarray*}
\item If $\chr k=p>0$ with $p\nmid d$ then
\begin{eqnarray*}
\rho(Z_{A,G}) &=& \begin{cases}
22 & \text{ if } p^\ell\equiv -1 \pmod{h_T} \text{ for some } \ell \\
22-\phi(h_T) & \text{ if } p^\ell\not\equiv -1 \pmod{h_T} \text{ for all } \ell \\
\end{cases} \\
\rho(Z_{A^T,G^T}) &=& \begin{cases}
22 & \text{ if } p^\ell\equiv -1 \pmod{h} \text{ for some } \ell \\
22-\phi(h) & \text{ if } p^\ell\not\equiv -1 \pmod{h} \text{ for all } \ell \\
\end{cases}
\end{eqnarray*}
\end{enumerate}
\end{thm}

\begin{proof}
Applying Lemma \ref{exactly-one-lemma} to $(A^T,G^T)$, we find that ${\bf j}_A$ is the only element of age 1 in $G$.  Now the order of ${\bf j}_{A}$ in $(\Z/d)^4$ clearly divides $h$, and by using the fact that the weights $q_i$ have no nontrivial common divisor, one may deduce that its order is exactly $h$.  Therefore we conclude by Proposition \ref{0-prop} that
\begin{equation}\label{d-exp}
\mathfrak I_d(0)\cap G = \set{n{\bf j}_{A} \ | \ 1\leq n\leq h, (n,h)=1}
\end{equation}
has cardinality $\phi(h)$.  We may apply a symmetric line of reasoning to describe the structure of $\mathfrak I_d(0)\cap G^T$.  Combining this with Theorem \ref{Kelly-thm} and Proposition \ref{p-prop}, we obtain the theorem.
\end{proof}

\begin{example}
This example is taken from \S4 of \cite{Kelly}.  Let
\[
A = \begin{bmatrix}
2 & 1 & 0 & 0 \\
0 & 2 & 1 & 0 \\
0 & 0 & 6 & 1 \\
0 & 0 & 0 & 7
\end{bmatrix},
\]
so that
\[
F_A = x_0^2x_1+x_1^2x_2+x_2^6x_3+x_3^7,
\]
which is quasihomogeneous of degree $h=7$ and weights ${\bf q} = (2,3,1,1)$.  Thus $\Aut(F_A)$ is a group of order $\abs{\det(A)} = 168$, and one may show (see \cite{Kelly}) that $\SL(F_A)$ has order 21.  Since $J_{F_A}=\inn{{\bf j}_{A}}$ has order $h=7$, it follows that $[\SL(F_A):J_{F_A}]=3$, and thus there are only two groups $G$ such that $J_{F_A}\subseteq G\subseteq \SL(F_A)$, namely $G=J_{F_A}$ and $G=\SL(F_{A})$.

Reading off the rows of $A^T$, we also have
\[
F_{A^T} = x_0^2+x_0x_1^2+x_1x_2^6+x_2x_3^7,
\]
which is quasihomogeneous of degree $h_T=8$ and weights ${\bf q_T} = (4,2,1,1)$.  In this case, $J_{F_{A^T}}$ is cyclic of order 8, $\SL(F_{A^T})$ is cyclic of order 24, and so, depending upon $G$, one either has $G^T=J_{F_{A^T}}$ or $G^T=\SL(F_{A^T})$.  The precise correspondence happens to be
\[
(J_{F_A})^T = \SL(F_{A^T}), \qquad SL(F_A)^T = J_{F_{A^T}},
\]
although for the purposes of computing Picard numbers, Theorem \ref{main-thm} says that the choice of $G$ is irrelevant.  If $Z_{A,G}$ and $Z_{A^T,G^T}$ are taken to be defined over $k$ of characteristic zero, then the first part of the theorem gives
\[
\rho(Z_{A,G}) = 22-\phi(h_T) = 22-\phi(8) = 18
\]
and
\[
\rho(Z_{A^T,G^T}) = 22-\phi(h) = 22-\phi(7) = 16.
\]

Upon computing
\[
A^{-1} = 
\begin{bmatrix}
1/2 & -1/4 & 1/24 & -1/168 \\
0 & 1/2 & -1/12 & 1/84 \\
0 & 0 & 1/6 & -1/42 \\
0 & 0 & 0 & 1/7
\end{bmatrix},
\]
we find that $d = 168$.  If $\chr k=p$ and $p\nmid d$ then, in this example, the assumption $p\nmid d$ guarantees that $p$ does not divide the weights in ${\bf q}$ and ${\bf q}_T$.  We also note that $p\nmid h$ and $p\nmid h_T$, since $h,h_T$ both divide $d$. Therefore $p\equiv b \pmod{h_T}$ for some $1\leq b\leq h_T$ such that $(b,h_T)=1$.  Since $h_T=8$, this gives $p\equiv 1,3,5, \text{ or } 7 \pmod{8}$.  Only in the case $p\equiv 7 \pmod{8}$ do we have $p^\ell\equiv -1 \pmod{8}$ for some $\ell$.  Therefore by Theorem \ref{main-thm}
\[
\rho(Z_{A,G}) = \begin{cases}
22 & \text{ if } p\equiv 7 \pmod{8} \\
18 & \text{ if } p\equiv 1,3,5 \pmod{8}.
\end{cases}
\]
On the other hand, we have $p\equiv 1,2,3,4,5,\text{ or } 6 \pmod{7}$, and so $p^\ell\equiv -1 \pmod{7}$ for some $\ell$ if and only if $p\equiv 3,5,6 \pmod{7}$.  Therefore
\[
\rho(Z_{A^T,G^T}) = \begin{cases}
22 & \text{ if } p\equiv 3,5,6 \pmod{7} \\
16 & \text{ if } p\equiv 1,2,4 \pmod{7}.
\end{cases}
\]
This explains the observations in \cite[p.62]{Kelly}.

%Upon computing
%\[
%A^{-1} = 
%\begin{bmatrix}
%1/2 & -1/4 & 1/24 & -1/168 \\
%0 & 1/2 & -1/12 & 1/84 \\
%0 & 0 & 1/6 & -1/42 \\
%0 & 0 & 0 & 1/7
%\end{bmatrix},
%\]
%we find that $d = 168$, which happens to equal $\det(A)$.

\end{example}

\bibliographystyle{math}
\bibliography{BHK-biblio}

\end{document}